\documentclass[a4paper,11pt]{article}

\usepackage{amsmath}
\usepackage{amsthm}
\usepackage{verbatim}
\usepackage{enumerate}
\usepackage[shortlabels]{enumitem}

\usepackage[top=3.5cm, bottom=3.5cm, left=2.5cm, right=2.5cm]{geometry}

\usepackage[utf8]{inputenc}
\usepackage[T1]{fontenc}
\usepackage[french,english]{babel} 
\usepackage{enumerate}

\usepackage{xcolor}
\usepackage[]{pdfpages}
 \usepackage{lineno}
\usepackage{titling}
\usepackage[colorlinks=false,linkcolor=blue]{hyperref}
\usepackage{csquotes}
\usepackage[nottoc, notlof, notlot, numbib]{tocbibind}
\usepackage[style=alphabetic,giveninits,doi=false,isbn=false,url=false,eprint=false,clearlang=true,sorting=nyt,maxbibnames=99]{biblatex}
\DeclareSourcemap{
      \maps[datatype=bibtex]{
      \map[overwrite=false]{
       \step[fieldsource=note]
       \step[fieldset=addendum, origfieldval, final]
       \step[fieldset=note, null]
       }
   }
   }

\usepackage{mathtools}
\usepackage{thmtools}
\usepackage{bbm}
\usepackage{amssymb}
\usepackage{mathrsfs}
\usepackage{float}
\usepackage{stmaryrd}
\usepackage[normalem]{ulem}

\usepackage{setspace}
\usepackage{array}
\usepackage{tabularx}
\usepackage{ltablex} 

\usepackage{todonotes}
\usepackage{mathtools}
\mathtoolsset{showonlyrefs}

\DeclareUnicodeCharacter{221E}{$\infty$}
\DeclareUnicodeCharacter{00B4}{\'}
\DeclareUnicodeCharacter{03BB}{$\lambda$}
\DeclareUnicodeCharacter{0393}{$\Gamma$}
\DeclareUnicodeCharacter{211D}{$\mathbb{R}$}
\DeclareUnicodeCharacter{2212}{-}

\newcommand{\R}{\mathbb{R}}

\newcommand{\E}{\mathbb{E}}
\newcommand{\ps}{\mathcal{P}} 

\renewcommand{\d}{\mathrm{d}}
\newcommand{\F}{\mathcal{F}}
\renewcommand{\P}{\mathbb{P}} 
\renewcommand{\H}{\mathcal{H}}
\renewcommand{\L}{\mathcal{L}} 

\newtheorem{theorem}{Theorem}[section]
\newtheorem{proposition}[theorem]{Proposition}
\newtheorem{lemma}[theorem]{Lemma}

\theoremstyle{definition}

\theoremstyle{remark}
\newtheorem{rem}[theorem]{Remark}
\newtheorem{example}[theorem]{Example}

\theoremstyle{plain}
\newtheorem{assumption}{Assumption}

\usepackage{authblk}

\title{\bf Optimal rate of convergence in the vanishing viscosity for  uniformly convex Hamilton-Jacobi equations
}

\begin{document}


\author[1]{Louis-Pierre \textsc{Chaintron}}
\author[2]{Samuel \textsc{Daudin}}
\affil[1]{\small DMA, École normale supérieure, Université PSL, CNRS, 75005 Paris, France}
\affil[2]{\small Université Paris Cité, Laboratoire Jacques-Louis Lions}

\date{}

\maketitle

\begin{abstract}
The purpose of this note is to provide an optimal rate of convergence in the vanishing viscosity regime for first-order Hamilton-Jacobi equations with uniformly convex Hamiltonian. 
We prove that for a globally Lipschitz-continuous and semiconcave  
terminal condition the rate is of order $O (\varepsilon \log \varepsilon)$, and we provide an example to show that this rate cannot be sharpened. 
This improves on the previously known rate of convergence $O(\sqrt{\varepsilon})$, which was widely believed to be optimal. 
Our proof combines techniques involving regularisation by sup-convolution with entropy estimates for the flow of a suitable version of the adjoint linearized equation.
The key technical point is an integrated estimate of the Laplacian of the solution against this flow. 
Moreover, we exploit the semiconcavity generated by the equation to handle less regular data in the quadratic case.
\end{abstract}

\tableofcontents

\section{Introduction}

Let us fix $T > 0$ and an integer $d \geq 1$. 
Given a Lipschitz-continuous and semiconcave function $g : \R^d \rightarrow \R$, we are interested in the $\varepsilon \rightarrow 0$ convergence of the Hamilton-Jacobi-Bellman equation (HJB)
\begin{equation}  \label{eq:HJBeps} 
\begin{cases}
- \partial_t \varphi^\varepsilon_t (x) + \H (x,\nabla\varphi^\varepsilon_t(x)) = \frac{\varepsilon}{2} \Delta \varphi^\varepsilon_t (x), \\
\varphi^\varepsilon_T (x) = g (x).
\end{cases}
\end{equation}
Under the standard assumptions gathered in \ref{ass:H} below,
it is a well-known result \cite{friedman1973cauchy,amann1978some} that \eqref{eq:HJBeps} has a unique $\mathcal{C}^{1,2}$ solution $\varphi^\varepsilon : [0,T] \times \R^d \rightarrow \R$, whose gradient is bounded by some $L > 0$ independent of $\varepsilon$:
\begin{equation} \label{eq:GradBound}
\forall \varepsilon \in (0,1/2], \qquad \sup_{(t,x) \in [0,T] \times \R^d} \vert \nabla \varphi^{\varepsilon}_t (x) \vert \leq L, 
\end{equation} 
see e.g. \cite{lions1982generalized,lions1984two}. 
From the fundamental works \cite{evans1980solving,crandall1983viscosity,crandall1984some,bardi1997optimal}, $\varphi^\varepsilon$ converges uniformly on compact sets towards the unique viscosity solution of \begin{equation} \label{eq:HJB0}
\begin{cases}
-\partial_t \varphi^0_t (x) + \H ( x, \nabla \varphi^0_t (x) ) = 0, \\
\varphi^0_T (x) = g (x).
\end{cases}
\end{equation}
In this setting, $\varphi^0_t$ inherits the gradient bound \eqref{eq:GradBound} from $\varphi^\varepsilon_t$ and is globally Lispchitz in time, so that $\varphi^0$ is Lebesgue-almost everywhere (a.e.) differentiable and \eqref{eq:HJB0} is satisfied a.e. 

A natural question is then to understand the speed of this approximation.  
The convergence rate below seems to be due to \cite{fleming1964convergence} and is known to be optimal if no convexity assumption is made on $H$.

\begin{theorem}[Sub-optimal rate] \label{thm:subopt}
There exists $C_{\mathrm{sqrt}} >0$ such that
\begin{equation}
\forall \varepsilon \in (0,1/2], \quad \sup_{(t,x) \in [0,T] \times \R^d} \vert \varphi^\varepsilon_t (x) - \varphi^0_t (x) \vert \leq C_{\mathrm{sqrt}} \sqrt{\varepsilon}. 
\end{equation}  
\end{theorem}

The main result of the present article is to show that semiconcavity for the terminal data $g$ together with the uniform convexity assumption 
\[ \exists \theta > 0: \forall (x,p) \in (\R^d)^2, \qquad \nabla^2_{pp} \H(x,p) \geq \theta \mathrm{Id}, \]
yield the following optimal improvement for the convergence rate.

\begin{theorem}[Optimal rate] \label{thm:opt}
Under \ref{ass:H}, there exists $C_{\log} >0$ such that
\begin{equation}
\forall \varepsilon \in (0,1/2], \quad \sup_{(t,x) \in [0,T] \times \R^d} \vert \varphi^\varepsilon_t (x) - \varphi^0_t (x) \vert \leq - C_{\log} \, \varepsilon \log \varepsilon.
\end{equation}  
Furthermore, there exists $C > 0 $ such that
\[ \forall (t,x) \in [0,T] \times \R^d, \quad \varphi^\varepsilon_t (x) - \varphi^0_t (x) \geq d \varepsilon \log \varepsilon -C \varepsilon, \]
where $C$ only depends on the regularity constants of $(g,\H)$. 
\end{theorem}

The proof of Theorem \ref{thm:opt} relies on the regularising effect of a suitable stochastic flow combined with the semiconcavity properties of the value functions.
The main technical novelty is Proposition \ref{pro:Regularising}-\ref{item:LapLow}, which leverages entropy bounds for estimating integrals of the kind $\int_{t}^T \int_{\R^d} \Delta \psi_s \d \mu_s \d s$ when $\mu_s$ is the solution of the Fokker-Planck equation $\partial_s \mu_s = \nabla \cdot [ \mu_s \nabla \psi_s] + \Delta \mu_s$.  
We provide an example in Section \ref{sec:Example}, which shows that the rate cannot be better than $O(\varepsilon \log \varepsilon)$, and that the constant cannot be lower than $(d-1)/2$. 
Our result demonstrates that uniformly convex Hamiltonians actually enhance the speed of convergence. 
To wit, in the  simplest setting without non-linearity, the optimal rate cannot be better than $\sqrt{\varepsilon}$.
The case of semiconcave data is handled by Theorem \ref{thm:optSC} in Section \ref{subsec:SC}. 
When the Hamiltonian is purely quadratic, Section \ref{subsec:NSC} studies the case of Lipschitz terminal data $g$ that may no more be semiconcave.

Interestingly, the leading order term is given by a universal dimensional constant.
Our result is reminiscent of the rate of convergence of the cost in the entropic approximation of optimal transport \cite{Carlier2023}.
The latter question is related to the small noise limit in the Schrödinger bridge problem \cite{leonard2013survey}, which can be rephrased as a stochastic control problem of planning type, the initial and terminal laws being imposed.
In contrast, our control interpretation of \eqref{eq:HJBeps} lets the terminal law free but adds a terminal cost, see Section \ref{subsec:controlsto}. 
Our result is also related to the probabilistic large deviation theory, since Theorem \ref{thm:opt} gives the optimal rate of convergence for the exponential moment $-\varepsilon \log \E [ e^{-g(\sqrt{\varepsilon} B_t) / \varepsilon} ]$ of a rescaled Brownian motion (Schilder's theorem), see \cite{feng2006large} for the use of \eqref{eq:HJBeps} in large deviation theory.

\subsection{A long standing history}

The convergence of \eqref{eq:HJBeps} towards \eqref{eq:HJB0} has been studied by many authors, among which \cite{Hopf1950ThePD,oleinik1963discontinuous,varadhan1966asymptotic}.
We also refer to \cite{kruzhkov1964cauchy,kuznetsov1964higher} for generalizations to systems of conservation laws.
Indeed, for $d=1$, $\nabla \varphi^0$ corresponds -- up to time-reversal -- to the unique entropy solution of the Burgers equation, which can also be obtained by vanishing viscosity methods \cite{kruvzkov1970first,lax2005hyperbolic}.

To our knowledge, the convergence rate $O(\sqrt{\varepsilon})$ was first obtained by \cite{fleming1964convergence} using a differential game approach. 
Within the framework of viscosity solutions, \cite{crandall1983viscosity,lions1984two} established it for general Hamiltonians using the doubling variables method. 
For some types of non-convex Hamiltonians, this rate is shown to be optimal in \cite{qian2024optimal}.
\cite{bardi1997optimal} further showed that this rate deteriorates if the terminal data is only Hölder-continuous. 
When the Hamiltonian is convex, a natural approach to these questions relies on the stochastic control interpretation of $\varphi^\varepsilon_t$, see Section \ref{subsec:controlsto} below.
Some other proofs leverage the nonlinear adjoint method introduced by \cite{evans2010adjoint,tran2011adjoint}.

When the Hamiltonian is uniformly convex -- as in \eqref{eq:HJBeps} --, better rates can be expected if the solution is semiconcave.
In this setting, for a semiconcave terminal condition, a $O(\varepsilon)$ upper bound is known to hold \cite{lions1982generalized,calder2018lecture,camilli2023quantitative}, see Lemma \ref{lem:UppOpt} below -- when $g$ is no more semiconcave, we further prove a  deteriorated version of this upper bound in Proposition \ref{pro:UppNSC}.
Using the non-linear adjoint method, \cite{tran2021hamilton} obtained a $O(\varepsilon)$ rate for suitable averages of $\varphi^\varepsilon - \varphi^0$.
Working in the torus, \cite{camilli2023quantitative} obtained a $O(\varepsilon)$ rate in $L^\infty ( L^1 )$.
See also \cite{Tang-Teng} for a similar one-dimensional result for conservation laws.
The global rate $O(\varepsilon \log \varepsilon)$ had already been obtained in dimension $d=1$ for purely quadratic Hamiltonians and Lipschitz terminal data, see \cite[Proposition 4.4]{qian2024optimal}, which relies on the Cole-Hopf formula \eqref{eq:HopfC} for $\varphi^\varepsilon$ and the Hopf-Lax formula \eqref{eq:HopfLax} for $\varphi^0$.
See also \cite{wang19981} for the analogous result for the viscous Burgers equation, as well as \cite{bressan2004convergence,bressan2012convergence} for $O(\sqrt{\varepsilon} \log \varepsilon)$ rates for the vanishing viscosity approximation of hyperbolic systems.
The works 
\cite{droniou2006fractal,goffi2024remarks} have studied regularising effects of \eqref{eq:HJBeps} when the usual Laplacian is replaced by a non-local fractional Laplacian.
More precisely, \cite{goffi2024remarks} computed optimal rates in this case and deduced a $O(\varepsilon)$ rate for \eqref{eq:HJBeps} in $L^\infty([0,T],L^p( \R^d))$ for every $1 < p < +\infty$.
Independently of our work and at the same time, the optimal $O(\varepsilon \log \varepsilon)$ rate has been established in the torus by \cite{CirGoffOpt} for a large class of convex Hamiltonians, using the nonlinear adjoint method.
To the best our knowledge, Theorem \ref{thm:opt} and \cite{CirGoffOpt} are the first results giving the optimal rate in every dimension, which has been an open question for a long time.

Since the early works of Fleming, it was realized that the rate of convergence deteriorates around the singularities of the solution to the first order Hamilton-Jacobi equation. In the special case where the terminal condition $g$ is $\mathcal{C}^{1,1}$ and convex, then the solution $\varphi_t^0$ is actually $\mathcal{C}^{1,1}$ for all $t \in [0,T]$, and a simple comparison argument shows that the rate of convergence is of order $O(\varepsilon)$. When $g$ is not convex, the solution can develop singularities. 
In this case, from \cite[Example 10.2']{fleming1971stochastic} in a similar setting, we cannot expect a global rate better than $O(\varepsilon\log\varepsilon)$. On the other hand, the solution is always semiconcave and singularities of semiconcave functions cannot be too numerous. The optimal rate $O(\varepsilon)$ can then  be achieved if we restrict ourselves to points $(t,x)$ in the \emph{strong regularity region}, which is known to be a dense open subset of $[0,T] \times \R^d$ \cite[Theorem 2]{fleming1964cauchy} and \cite[Chapters 4-5]{cannarsa2004semiconcave}. We refer to \cite{fleming1964cauchy,fleming1971stochastic} for further results and an asymptotic expansion of $\varphi^\varepsilon_t$ w.r.t. $\varepsilon$ within the strong regularity region.
These works rely on a probabilistic stochastic control approach. See \cite{fleming1986asymptotic} for an analogous result  using viscosity solution methods.

\subsection{Stochastic control interpretation} \label{subsec:controlsto}

In this section, we briefly sketch the probabilistic interpretation of \eqref{eq:HJBeps}-\eqref{eq:HJB0} when $\H$ is convex, in the spirit of the first derivation \cite{fleming1964convergence} of the $O(\sqrt{\varepsilon})$ rate.
Let us fix a filtered probability space $(\Omega,\F,(\F_t)_{0 \leq t \leq T}, \P)$. For $\varepsilon > 0$, we consider the controlled dynamics
\begin{equation} \label{eq:trajeps}
 X^{\varepsilon,\alpha}_s := x + \int_t^s \alpha_s \d s + \sqrt{\varepsilon} B_s, \quad t \leq s \leq T, 
\end{equation}
where $(B_s)_{t \leq s \leq T}$ is a $(\F_t)_{t \leq s \leq T}$-Brownian motion starting at $B_t = 0$.
The control $\alpha = (\alpha_s)_{t \leq s \leq T}$ is a stochastic process, which is progressively measurable and square-integrable.
Leveraging the regularity of $\varphi^\varepsilon$ and Bellman's dynamic programming principle, it is well-known that $\varphi^\varepsilon$ corresponds to the the value function of a stochastic control problem,
\begin{equation} \label{eq:Vcontroleps}
\varphi^\varepsilon_t ( x ) = \inf_{\alpha, X^{\varepsilon,\alpha}_t = x} \E \bigg[ \int_t^T \L(  X^{\varepsilon,\alpha}_s, \alpha_s ) \d s + g ( X^{\varepsilon,\alpha}_T ) \bigg],
\end{equation} 
where the Lagrangian $\L$ is defined through the Legendre transform
\[ \L(x,\alpha) := \sup_{p \in \R^d} - \alpha \cdot p - \H(x,p),  \]
see e.g. the textbook \cite{fleming2006controlled} on stochastic control.

Alternatively, this can be written as the PDE control problem
\[ \inf_{h} \int_t^T \int_{\R^d} \L( x, h_s (x)) \d \mu^{\varepsilon,h}_s (x) \d s, \quad \partial_s \mu^{\varepsilon,h}_s = \nabla \cdot [ - \mu^{\varepsilon,h}_s h_s + \tfrac{\varepsilon}{2} \nabla \mu^{\varepsilon,h}_s ], \quad \mu^{\varepsilon,h}_t = \delta_x,   \]
the control parameter being the vector field $h : [t,T] \times \R^d \rightarrow \R^d$.
Comparing with \eqref{eq:Vcontroleps}, this corresponds to $\alpha_s = h_s ( X^{\varepsilon,\alpha}_s)$ and $\mu^{\varepsilon,h}_s = \mathrm{Law}(X^{\varepsilon,\alpha}_s)$. 

The unique optimal curve for \eqref{eq:Vcontroleps} is then the pathwise unique solution of the stochastic differential equation (SDE)
\begin{equation} \label{eq:opteps}
\overline{X}^\varepsilon_s = x - \int_t^s \nabla_p \H( \overline{X}^\varepsilon_s, \nabla \varphi^\varepsilon_s ( \overline{X}^\varepsilon_s ) ) \d s + \sqrt{\varepsilon}  B_s, 
\end{equation} 
so that $\nabla \varphi^\varepsilon$ provides the optimal control in feed-back form.

Sending $\varepsilon \rightarrow 0$, $\varphi^0_t (x)$ is also the value function of a deterministic control problem,
\begin{equation} \label{eq:Vcontrol0}
\varphi^0_t ( x ) = \inf_{\alpha, X^{0,\alpha}_t = x} \, \int_t^T \L( X^{0,\alpha}_s, \alpha_s ) \d s + g ( X^{0,\alpha}_T ),
\end{equation} 
where we here minimise over \emph{deterministic controls} $\alpha \in L^2 ( (t,T), \R^d )$.
However, this problem does not have a unique optimal solution in general, unless $L$ and
$g$ are strictly convex functions.
This lack of uniqueness is direclty related to the lack of regularity of $\varphi^0_t$, which is not $\mathcal{C}^1$ in general.
At this stage, following \cite{fleming1964convergence}, a direct coupling argument of the value functions \eqref{eq:Vcontroleps}-\eqref{eq:Vcontrol0} yields the suboptimal rate $O(\sqrt{\varepsilon})$.

\subsection{Comparison with mean-field control} \label{subsec:MFC}

The search for the optimal rate of convergence in the vanishing viscosity problem has sparked renewed interest in recent years, partly due to analogous questions arising in mean-field optimal control theory. Mean-field optimal control theory focuses on optimal control problems for large populations of $N \geq 1$ interacting particles. A central question in this theory is to quantitatively understand the regime $N \rightarrow +\infty$. This can be approached using PDE techniques, particularly at the level of the value functions for the different problems. In certain settings, where each particle is subject to its own idiosyncratic Brownian noise, the pre-limit problem can be viewed as a noisy perturbation of the limit problem, which is formulated over a space of probability measures. In this context, the convergence problem in mean-field control theory can be seen as an infinite-dimensional analog of the vanishing viscosity problem, where the viscosity intensity is given by $ \varepsilon = 1/N$. This connection has inspired a recent line of research applying well-established techniques from the vanishing viscosity problem to the convergence problem in mean-field control theory, as seen in \cite{cdjs2023,ddj2023,cjms2023,cecchin2025quantitativeconvergencemeanfield}. The analogy extends further, as the vanishing viscosity problem can itself be viewed as a special case of the mean-field problem when the interaction is solely through the mean position of the population, as demonstrated in \cite[Example 2.13]{ddj2023}. Consequently, the optimal rate of convergence in the vanishing viscosity regime provides a lower bound on the optimal rate in the mean-field regime. In particular, Theorem \ref{thm:opt} shows that the convergence rate in \cite[Theorem 2.7]{ddj2023} -- in the so-called \textit{regular} case where the terminal cost is Lipschitz and semi-concave with respect to sufficiently weak metrics -- could be better than $N^{-1/2}$ but not better than $N^{-1} \log(N)$. Adapting the techniques presented in this paper to the mean-field setting is an exciting yet undoubtedly challenging prospect, which we leave for future work.

\subsection{Extensions and open questions}

For the clarity of exposition, we restricted ourselves to the setting of Assumption \ref{ass:H}.
Under slight modifications, we could handle further dependences of $\H$ on $t$, other growth conditions or uniformly elliptic diffusion matrix, see the discussion after \ref{ass:H} below.

As noticed in \cite{CirGoffOpt}, the result proved for \eqref{eq:HJBeps} can be transferred to the stationary case
\[ \beta \varphi^\varepsilon (x) + \H ( x , \nabla \varphi^\varepsilon (x) ) = \frac{\varepsilon}{2} \Delta \varphi^\varepsilon (x), \]
for $\beta >0$,
using the change of variable $(t,x) \mapsto e^{\beta t} \varphi^\varepsilon (x)$ exploited in \cite{tran2011adjoint}. 
\cite[Corollary 3.4]{CirGoffOpt} further deduces convergence rates for $\lVert \nabla \varphi^\varepsilon_t - \nabla \varphi^0_t \rVert_{L^2}$ on compact sets, with application to convergence of parabolic systems of conservation laws towards hyperbolic ones. 

The optimal rate given by Theorem \ref{thm:opt} can be applied to many problems using the vanishing viscosity procedure, like vanishing viscosity for mean-field games \cite{tang2025policy} or regularisation by viscosity for numerical schemes \cite{tang2023convergence}.
Let us also mention an intriguing analogy with convergence rates for finite-difference monotone approximations of HJB equations: the method of ``shaking coefficients'' \cite{krylov2000rate,barles2007error} produces $O(\sqrt{h})$ convergence rates in the mesh-size, the lack of regularity of \eqref{eq:HJB0} hindering $O(h)$ rates.
Although the setting is different, these rates are very reminiscent of ours, raising the question of an improved optimal rate.

A few other perspectives are listed below.
\begin{itemize}
\item A natural question is to extend our results to settings with boundary questions as in \cite{han2022remarks,hasebe2024construction}.
This would enable a further comparison with the results of \cite{fleming1971stochastic,fleming1986asymptotic}.
Homogenization problems in the spirit of \cite{qian2024optimal} would be interesting too.
\item The fundamental obstruction to the smooth-setting $O( \varepsilon )$ rate is the lack of regularity of $\varphi^0_t$.
It would be illuminating to relate the $O( \varepsilon \log \varepsilon )$ rate (and the pre-factor $d$) to the singularities of $\nabla \varphi^0_t$. 
Such a study will likely require a thorough analysis of the singularities in the spirit of \cite{cannarsa2004semiconcave}. 

\item We eventually refer to Section \ref{subsec:MFC} for the promising perspective of extending our results to the mean-field control setting in the spirit of \cite{ddj2023}. 
\end{itemize}

\section{Uniformly convex Hamiltonians}

We work under the following assumptions, which are standard in viscosity solution theory.

\begin{assumption} \label{ass:H}
$\phantom{a}$
\begin{enumerate}[label=(\roman*),ref=(\roman*)]
    \item\label{item:ATerm} Terminal data: $g$ is Lipschitz-continuous and semiconcave.
    \item\label{item:ALip} Regularity of the Hamiltonian: there exist $c, C > 0$ such that
    \[ \forall (x,p) \in (\R^d )^2, \qquad - C + c \vert p \vert^2 \leq \H(x,p) \leq C + c^{-1} \vert p \vert^2 \quad \text{and} \quad \vert \nabla_x \H(x,p) \vert \leq C [ \vert p \vert + 1 ]. \]
    \item\label{item:ASC} Further regularity: for every $R > 0$, there exist $C_R > 0$ such that 
    \[ \forall (x,p) \in (\R^d)^2, \qquad \vert p \vert \leq R \quad \Rightarrow \quad \vert \nabla^{2}_{xx} \H(x,p) \vert  + \vert \nabla^2_{xp} \H(x,p) \vert \leq C_R.  \]
    \item\label{item:AUC} Uniform convexity: for every $R > 0$, there exist $\theta_R, \Theta_R > 0$ such that 
    \[ \forall (x,p) \in (\R^d)^2, \qquad \vert p \vert \leq R \quad \Rightarrow \quad \theta_R \mathrm{Id} \leq \nabla^{2}_{pp} \H(x,p) \leq \Theta_R \mathrm{Id}.  \]
\end{enumerate}
\end{assumption}

These assumptions may not be optimal, but \ref{ass:H} is a convenient framework for obtaining Lipschitz and semiconcavity estimates for \eqref{eq:HJBeps} without excessive technicalities.
The key assumption for the $O(\varepsilon \log \varepsilon)$ rate is the uniform convexity \ref{item:AUC}.
Assumption \ref{item:ALip} implies the comparison principle for \eqref{eq:HJB0} and the gradient bound \eqref{eq:GradBound}, whereas \ref{item:ATerm}-\ref{item:ASC} are sufficient conditions for propagating semiconcavity along the equation.
Thus, Assumption \ref{ass:H} guarantees the existence of $L,\lambda > 0$ such that
\begin{equation}  \label{eq:Grad+SC}
\forall \varepsilon \in (0,1/2], \forall (t,x) \in [0,T] \times \R^d, \qquad \vert \nabla \varphi^{\varepsilon}_t (x) \vert \leq L \quad \text{and} \quad \nabla^2 \varphi^\varepsilon_t (x) \leq \lambda \mathrm{Id}.
\end{equation} 
A detailed proof is given in e.g. \cite[Section 3.1]{cdjs2023} in a more general setting.
The estimates \eqref{eq:Grad+SC} extend to $\varepsilon = 0$, proving Lipschitz-continuity and semiconcavity for $\varphi^0_t$.

In the following, the needed properties are essentially \ref{item:AUC} and \eqref{eq:Grad+SC}, and our proof could be adapted to any set of assumptions that provide them.
For instance, we could allow for time-dependent Hamiltonians provided uniform bounds on $\partial_t \H$. 
Non-constant uniformly elliptic diffusion matrices could be covered as well. 
We restricted ourselves to the framework of \ref{ass:H} to favor the clarity of exposition.

\begin{example}[Quadratic case]
An important example of Hamiltonian satisfying \ref{ass:H} is $\H(x,p) := - b(x) \cdot p + \tfrac{1}{2} \vert p \vert^2$, for some bounded smooth vector field $b : \R^d \rightarrow \R^d$ whose first and second derivatives are bounded.
In this particular case, $b$ itself need not being bounded if its derivatives are bounded, because \eqref{eq:Grad+SC} is a well-known result in this setting.
\end{example}

\subsection{Proof outline}

\label{ssec:outline}

The lower bound on $\varphi^0 - \varphi^\varepsilon$ is a well-known consequence of the semiconcavity \eqref{eq:Grad+SC}, see Lemma \ref{lem:UppOpt} below.
To explain how we obtain the upper bound, consider, for simplicity, the case of purely quadratic Hamiltonian $\H (p) = \vert p\vert^2/2$. 
The first step is to regularise $\varphi^0$ into $\varphi^{0,\delta}$ for some parameter $\delta>0$ by sup-convolution, see Lemma \ref{lem:supconv}.
By making the difference between the equations satisfied by $\varphi^{0,\delta}$ and $\varphi^{\varepsilon}$ we find that $\varphi^{0,\delta} - \varphi^{\varepsilon}$ is a sub-solution to 
\begin{equation} 
\bigl \{-\partial_s + b_s^{\varepsilon,\delta} \cdot \nabla - \frac{\varepsilon}{2} \Delta \bigr \} ( \varphi_s^{0,\delta} - \varphi_s^{\varepsilon}) \leq - \frac{\varepsilon}{2} \Delta \varphi_s^{0,\delta}, \quad \quad \varphi_T^{0,\delta} - \varphi_T^{\varepsilon} \leq C \delta 
\label{eq:subsoloutline20/04}
\end{equation}
for some $C>0$ depending on the terminal condition $g$ and for the vector field $b^{\varepsilon,\delta}$ defined by
\[ b_s^{\varepsilon,\delta}(x) := \frac{\nabla \varphi_s^{0,\delta}(x) + \nabla \varphi_s^{\varepsilon}(x) }{2},  \quad \quad (s,x) \in [0,T] \times \R^d.\]
Now we fix $(t,x) \in [0,T] \times \R^d$ and we introduce $(\mu_s^{\varepsilon,\delta})_{s \in [t,T]}$ the solution to the forward Fokker-Planck equation
\begin{equation} 
\partial_s \mu_s^{\varepsilon,\delta} - \nabla \cdot (b_s^{\varepsilon,\delta}\mu_s^{\varepsilon,\delta}) - \frac{\varepsilon}{2} \Delta \mu_s^{\varepsilon,\delta} = 0 \quad \mbox{ in } (t,T) \times \R^d, \quad \mu_t^{\varepsilon,\delta} = \delta_x.
\label{eq:FPEIntro20/04}
\end{equation}
Integrating \eqref{eq:subsoloutline20/04} against $\mu_s^{\varepsilon,\delta}$ we obtain
\begin{equation} 
\varphi_t^{0,\delta} (x) - \varphi_t^{\varepsilon} (x) \leq C \delta - \frac{\varepsilon}{2} \int_t^T \int_{\R^d} \Delta \varphi_s^{0,\delta}(x) \d \mu_s^{\varepsilon,\delta}(x) \d s. 
\label{eq:estimateondiff20/04}
\end{equation}
Our main goal is then to estimate the integral in the right-hand side. 
The sup-convolution regularisation gives us the upper bound $-\Delta \varphi_s^{0,\delta} \leq \delta^{-1}$. 
If we used this estimate for all $s \in [t,T]$ and then optimize over $\delta$ we would recover the well-known rate of convergence of order $\sqrt{\varepsilon}$. Although we rely on this argument for $s \in [t,t+\tau]$ with $\tau >0$ small, our main innovation is to use the regularising properties of $\mu_s^{\varepsilon,\delta}$ to find a better upper bound on $-\int_{t+\tau}^T \int_{\R^d} \Delta \varphi_s^{0,\delta}(x)\d\mu_s^{\varepsilon,\delta}(x) \d s$ for $\tau \in (0,T-t)$. Recalling the definition of $b^{\varepsilon,\delta}$ we find that, up to a term of order $1$, $-\frac{1}{2} \Delta \varphi_s^{0,\delta} \leq - \nabla \cdot b_s^{\varepsilon,\delta}$ and then the key observation is that, differentiating the entropy $\mu \mapsto \int_{\R^d} \log \mu(x) \d \mu(x)$ along the flow of the Fokker-Planck equation and then integrating back in time we get,
\begin{align*} 
- \int_{t+\tau}^T \int_{\R^d} \nabla \cdot b_s^{\varepsilon,\delta}(x) \d \mu_s^{\varepsilon,\delta}(x) \d s &= \int_{\R^d} \log \mu^{\varepsilon,\delta}_{t +\tau}(x) \d\mu_{t+\tau}^{\varepsilon,\delta}(x) - \int_{\R^d} \log \mu_T^{\varepsilon,\delta}(x) \d\mu_T^{\varepsilon,\delta}(x)   \\
&  -\frac{\varepsilon}{2} \int_{t+\tau}^T \int_{\R^d} \frac{|\nabla \mu_s^{\varepsilon,\delta}(x)|^2}{ \mu_s^{\varepsilon,\delta}(x)} \d x \d s. 
\end{align*}
The leading order term is the first one in the right-hand side and we are reduced to estimating the blow up, near the initial time, of the entropy of the solution to the Fokker-Planck equation. 
There are then several ways of showing that 
\[ \int_{\R^d} \log \mu_{t + \tau}^{\varepsilon,\delta}(x) \d \mu_{t +\tau}^{\varepsilon,\delta}(x) \leq - \frac{d}{2} \log(2 \pi \varepsilon \tau) + \frac{C \tau}{2 \varepsilon},   \]
for some $C>0$. Combining everything together and taking $ \delta = \tau = \varepsilon$ gives the result.

Now we come to the detailed justification of the proof.

\subsection{Semiconcave setting}  \label{subsec:SC}

Leveraging the semiconcavity \eqref{eq:Grad+SC}, the following estimate is a well-known fact. 

\begin{lemma}[Upper bound] \label{lem:UppOpt}
For every $\varepsilon \in (0,1/2]$,
\begin{equation*} \label{eq:KnownUpper}
\forall (t,x) \in [0,T] \times \R^d, \quad \varphi^{\varepsilon}_t ( x) - \varphi^{0}_t ( x) \leq \frac{(T - t )d \lambda}{2} \varepsilon.
\end{equation*}
\end{lemma}

\begin{proof}
By semiconcavity, $\Delta \varphi^\varepsilon_t \leq d\lambda$, so that $(t,x) \mapsto \varphi^\varepsilon (t,x) - \varepsilon(T -  t) d\lambda / 2$ is a viscosity sub-solution of \eqref{eq:HJB0}.
By comparison with the viscosity (super-)solution $\varphi^0$, the result follows. 
\end{proof}

The main difficulty to prove the lower bound is the lack of semiconvexity for $\varphi^0_t$.
To circumvent this, we regularise $\varphi^0_t$ using the standard sup-convolution \cite{lasry1986remark}, for $\delta >0$,
\[ \varphi^{0,\delta}_t (x) := \sup_{y \in \R^d} \varphi^{0}_t (y) - \frac{1}{2 \delta} \vert x - y \vert^2. \]
The following lemma recalls known properties of sup-convolutions, see e.g. \cite{calder2018lecture,tran2021hamilton}.
Since $\nabla \varphi^\varepsilon$ is uniformly bounded, \ref{ass:H}-\ref{item:ALip} allows us to assume, up to redefining $\H$, that $x \mapsto \H(x,p)$ is globally Lipschitz uniformly in $p$.
As a consequence we can fix some $R \geq L$ in \ref{ass:H}, and we will simply write $(\theta,\Theta)$ instead of $(\theta_R,\Theta_R)$.

\begin{lemma}[Sup-convolution]$\label{lem:supconv} \phantom{a}$
\begin{enumerate}[label=(\roman*),ref=(\roman*)]
\item\label{item:supReg} $\varphi^{0,\delta}_t$ is $L$-Lipschitz, $\lambda$-semiconcave, and $-1/\delta$-semiconvex. In particular, $\varphi^{0,\delta}_t$ is $\mathcal{C}^{1,1}$.
\item\label{item:supAppr} $\forall (t,x) \in [0,T] \times \R^d, \quad 0 \leq \varphi^{0,\delta}_t (x) - \varphi^0_t (x) \leq 2 L \delta$.
\item\label{item:subsol} $\varphi^{0,\delta}_t$ inherits the sub-solution property for \eqref{eq:HJB0}: there exists $C_{\mathrm{sqrt}} >0$ such that for a.e. $t \in [0,T]$,
\begin{equation} \label{eq:convolSub}
\forall x \in \R^d, \quad - \partial_t \varphi^{0,\delta}_t (x) + \H (x, \nabla \varphi^{0,\delta}_t (x)) \leq 2 \delta C_{\mathrm{sub}},
\end{equation} 
recalling that $\varphi^{0,\delta}$ is locally Lipschitz in $t$ and $\mathcal{C}^{1,1}$ in $x$. 
\end{enumerate}
\end{lemma}

Let us fix $(t,x) \in [0,T] \times \R^d$. We introduce
\begin{equation} \label{eq:deDrift}
b^{\varepsilon,\delta}_s (x) := \int_0^1 \nabla_p \H ( x , r \nabla \varphi^\varepsilon_s (x) + (1-r) \nabla \varphi^{0,\delta}_s (x)  ) \d r, 
\end{equation} 
as well as the flow of measures $(\mu^{\varepsilon,\delta,x}_s)_{t \leq s \leq T}$ solution of
\begin{equation} \label{eq:Green}
\partial_s \mu^{\varepsilon,\delta,x}_s = \nabla \cdot \big[ \mu^{\varepsilon,\delta,x}_s b^{\varepsilon,\delta}_s + \frac{\varepsilon}{2} \nabla \mu^{\varepsilon,\delta,x}_s \big], \qquad \mu^{\varepsilon,\delta,x}_t = \delta_x. 
\end{equation} 
When there is no ambiguity, we will simply write $\mu^{\varepsilon,\delta}_s$ for $\mu^{\varepsilon,\delta,x}_s$.
This corresponds to the fundamental solution of the parabolic equation \eqref{eq:Green}.
Since $b^{\varepsilon,\delta}_s$ is locally Lipschitz, well-posedness for $(\mu^{\varepsilon,\delta}_s)_{t \leq s \leq T}$ is given by the classical work \cite{aronson1968non}, or as a particular case of \cite[Section 9.4]{bogachev2022fokker}.
Moreover, $\mu^{\varepsilon,\delta}_s$ has a positive density for $s \in (t,T]$.

\begin{proposition}[Key regularising effect] \label{pro:Regularising}
For every $\tau \in (0,T-t]$, $\mu^{\varepsilon,\delta,x}_{t+\tau} \in \ps ( \R^d )$ has a positive density, still denoted by $\mu_{t +\tau}^{\varepsilon,\delta,x} \in L^1 ( \R^d )$, which satisfies
\begin{enumerate}[label=(\roman*),ref=(\roman*)]
\item\label{item:RegEnt} 
$ - \infty < \int_{\R^d} \log \mu^{\varepsilon,\delta,x}_{t+\tau} \d \mu^{\varepsilon,\delta,x}_{t+\tau} \leq- \frac{d}{2} \log ( 2 \pi \varepsilon \tau ) + \frac{\tau}{2 \varepsilon} L^2$,
\item\label{item:LapLow} $\frac{1}{2} \int_{t+\tau}^T \int_{\R^d} \Delta \varphi^{0,\delta}_s \d \mu^{\varepsilon,\delta, x}_s \d s \geq \frac{d}{2 \theta} \log ( 2 \pi \varepsilon \tau ) - \frac{\tau}{2 \varepsilon \theta} L^2 -C$,
\end{enumerate} 
where $C$ is a constant that only depends on $(d,T,\theta,\Theta,L,\lambda)$ given by \ref{ass:H} for some fixed $R \geq L$.
\end{proposition}

The entropy bound \ref{item:RegEnt} can be obtained through several ways. A natural approach is to leverage known Gaussian bounds for Green functions \cite{aronson1968non,sheu1991some}.
In \cite[Lemma 4.2]{cirant2025convergence}, this bound is obtained as a consequence of \cite[Corollary 7.2.3]{bogachev2022fokker}.
The following short proof is probabilistic and was suggested to the first author by Daniel Lacker. In the next proof we use the notation $H( \cdot | \cdot)$ for the relative entropy between two probability measures, namely, for a Polish space $E$ and two probability measures $\mu, \nu$ over $E$ we have
\begin{equation}
    H(\mu \vert \nu) := 
\begin{cases} 
\int_E \log \tfrac{\d \mu}{\d \nu} \d \mu & \mbox{ if } \frac{\d \mu}{\d\nu} \mbox{ exists, } \\
+ \infty & \mbox{ otherwise,}
\end{cases}
\end{equation}
where $\frac{\d \mu}{\d\nu}$ stands for the Radon-Nikodym derivative of $\mu$ with respect to $\nu$ when it exists.

\begin{proof}
To alleviate notations, we omit the superscript $x$ in this proof.
\medskip

\ref{item:RegEnt} Let $w^{\varepsilon}_{[t,t+\tau]} \in \ps ( \mathcal{C} ( [t,t+\tau] , \R^d ))$ denote the path law of $( x + \sqrt{\varepsilon} B_{s-t} )_{t \leq s \leq t + \tau}$, $(B_s)_{s \geq 0}$ being a standard Brownian motion.
The marginal law $w^{\varepsilon}_s$ at time $s$ is Gaussian centered at $x$ with variance $\varepsilon (s-t)$.
Let $\mu^{\varepsilon,\delta}_{[t,t+\tau]} \in \ps ( \mathcal{C} ( [t,t+\tau] , \R^d ))$ denote the path law of the Markov process with generator $-b_s \cdot \nabla + \tfrac{\varepsilon}{2} \Delta$, starting from $x$ at time $t$.
The marginal law of this process at time $s$ is $\mu^{\varepsilon,\delta}_s$ given by \eqref{eq:Green}.
We use the Girsanov transform to compute the pathwise relative entropy
\[ H ( \mu^{\varepsilon,\delta}_{[t,t+\tau]} \vert w^{\varepsilon}_{[t,t+\tau]} ) = \frac{1}{2\varepsilon} \int_{t}^{t+\tau} \int_{\R^d} \vert b^{\varepsilon,\delta}_s \vert^2 \d \mu^{\varepsilon,\delta}_s \d s \leq \frac{\tau}{ 2 \varepsilon} L^2. \]
From the contraction property of relative entropy \cite[Theorem D.13]{dembo2009large}, we get $H ( \mu^{\varepsilon,\delta}_{t+\tau} \vert w^{\varepsilon}_{t+\tau} ) \leq H ( \mu^{\varepsilon,\delta}_{[t,t+\tau]} \vert w^{\varepsilon}_{[t,t+\tau]} )$, so that
\begin{align*} 
\int_{\R^d} \log \mu^{\varepsilon,\delta}_{t+\tau} \d \mu^{\varepsilon,\delta}_{t+\tau} &= H( \mu_{t + \tau}^{\varepsilon,\delta} | w_{t+ \tau}^{\varepsilon}) + \int_{\R^d} \log w^{\varepsilon}_{t+\tau} \d \mu^{\varepsilon,\delta}_{t+\tau} \\
&\leq \frac{\tau}{2 \varepsilon} L^2 + \int_{\R^d} \log w^{\varepsilon}_{t+\tau} \d \mu^{\varepsilon,\delta}_{t+\tau}  \leq \frac{\tau}{2 \varepsilon} L^2 - \frac{d}{2} \log ( 2 \pi \varepsilon \tau ), 
\end{align*}
using the pointwise bound $\log w^{\varepsilon}_{t + \tau}(x) \leq -\frac{d}{2} \log(2\pi \varepsilon \tau) $ for the Gaussian density on $\R^d$.
\medskip

\ref{item:LapLow} Let us first assume that $\mu^{\varepsilon,\delta}$ is in $\mathcal{C}^{1,2} ( (0,T) \times \R^d)$. Since $\mu^{\varepsilon,\delta}$ is positive and satisfies the Fokker-Planck equation \eqref{eq:Green}, differentiating the entropy classically yields 
\begin{equation} \label{ItoLog}
\int_{\R^d} \log \mu^{\varepsilon,\delta}_T \d \mu^{\varepsilon,\delta}_T - \int_{\R^d} \log \mu^{\varepsilon,\delta}_{t+\tau} \d \mu^{\varepsilon,\delta}_{t+\tau} =  \int_{t+\tau}^T \int_{\R^d} \nabla \cdot b^{\varepsilon,\delta}_s \d \mu^{\varepsilon,\delta}_s \d s - \frac{\varepsilon}{2} \int_{t+\tau}^T \int_{\R^d} \vert \nabla \log \mu^{\varepsilon,\delta}_s \vert^2 \d \mu^{\varepsilon,\delta}_s \d s, 
\end{equation}
using integration by parts and the finiteness of the integrated Fisher information given by e.g. \cite[Theorem 7.4.1]{bogachev2022fokker}.
We then compute
\begin{align*}
&\nabla \cdot b^{\varepsilon,\delta}_s (x) = \int_0^1 [\nabla_x \cdot \nabla_p \H] ( x , r \nabla \varphi^\varepsilon_s (x) + (1-r) \nabla \varphi^{0,\delta}_s (x)  ) d r \\
&+ \int_0^1 \nabla^2_{pp} \mathcal{H} ( x , r \nabla \varphi^\varepsilon_s (x) + (1-r) \nabla \varphi^{0,\delta}_s (x)  ) \cdot [  r \nabla^2 \varphi^\varepsilon_s (x) + (1-r) \nabla^2 \varphi^{0,\delta}_s (x) ] dr
\end{align*} 
The second integrand involves a scalar product of Hessians.
Using the uniform convexity \ref{ass:H}-\ref{item:AUC} and the semiconcavity \eqref{eq:Grad+SC},
a standard linear algebra result \cite[Lemma A.1]{chaintron2023existence} yields
\[ \nabla^2_{pp} \H \cdot \nabla^2 [ r \varphi^\varepsilon_s + (1-r) \varphi^{0,\delta}_s ] \leq \theta \Delta [ r \varphi^\varepsilon_s + (1-r) \varphi_s^{0,\delta} ] + \lambda [ \mathrm{Tr}[ \nabla^2_{pp} \H] - d \theta ], \]
omitting the arguments of $\H$ for the sake of readability. Integrating the above inequality yields
\[ \frac{\theta}{2} \int_{t+\tau}^T \int_{\R^d} \Delta \varphi^{0,\delta}_s \d \mu^{\varepsilon,\delta}_s \d s \geq \int_{t+\tau}^T \int_{\R^d} \nabla \cdot b^{\varepsilon,\delta}_s \d \mu^{\varepsilon,\delta}_s \d s - \frac{\theta}{2} \int_{t+\tau}^T \int_{\R^d} \Delta \varphi^{\varepsilon}_s \d \mu^{\varepsilon,\delta}_s \d s - \lambda [ d\Theta - d \theta] - C, \]
where $\Theta$ is given by the bound \ref{ass:H}-\ref{item:AUC} on $\nabla^2_{pp} \H$, and $C$ results from \eqref{eq:GradBound} and the bound \ref{ass:H}-\ref{item:ASC} on $\nabla^2_{xp} \H$.
Since $\Delta \varphi^{\varepsilon,\delta}_t \geq - \lambda d$ using the semiconvavity \eqref{eq:Grad+SC}, plugging this in \eqref{ItoLog} yields
\[ \frac{\theta}{2} \int_{t+\tau}^T \int_{\R^d} \Delta \varphi^{0,\delta}_s \d \mu^{\varepsilon,\delta}_s \d s \geq - C + \int_{\R^d} \log \mu^{\varepsilon,\delta}_{T} \d \mu^{\varepsilon,\delta}_{T} - \int_{\R^d} \log \mu^{\varepsilon,\delta}_{t+\tau} \d \mu^{\varepsilon,\delta}_{t+\tau}, \]
up to modifying the constant $C$ -- we also got rid of the non-positive integrated Fisher information term.
We now lower bound the last term on the r.h.s. using \ref{item:RegEnt}. For the middle term, denoting the Gaussian density $\mathcal{N}( x ,{ \mathrm{Id}})$ by $\gamma$,
\[ \int_{\R^d} \log \mu^{\varepsilon,\delta}_T \d \mu^{\varepsilon,\delta}_{T} = H ( \mu^{\varepsilon,\delta}_{T} \vert \gamma ) + \int_{\R^d} \log \gamma \d \mu^{\varepsilon,\delta}_{T} \geq - \frac{d}{2}\log( 2 \pi ) - \frac{1}{2} \int_{\R^d} ( y - x )^2  \d \mu^{\varepsilon,\delta}_{T} ( y ). \]
Furthermore, using the Fokker-Planck equation \eqref{eq:Green} and Jensen's inequality,
\[ \frac{\d}{\d s} \int_{\R^d} ( y - x )^2 \d \mu^{\varepsilon,\delta}_{s} ( y ) \leq 2 L \bigg( \int_{\R^d} ( y - x )^2 \d \mu^{\varepsilon,\delta}_{s} (y) \bigg)^{1/2} + d \varepsilon, \]
so that $\int_{\R^d} ( y - x )^2 \d \mu^{\varepsilon,\delta}_{T} (y) \leq C$ by integrating the derivative of a square-root, for some $C>0$ depending on $T,L$ and $d$. Gathering terms yields the r.h.s of \ref{item:LapLow}.

In general $\mu^{\varepsilon,\delta}$ is not $\mathcal{C}^{1,2}$, but we can regularise $b^{\varepsilon,\delta}$ and use a standard approximation argument.
Indeed, the Green function $\mu^{\varepsilon,\delta}$ is in $L^2 ( [t+\tau,T], H^1 ( \R^d))$ from \cite[Theorem 9]{aronson1968non}, and can be approximated weakly in $L^2 ( [t+\tau,T], H^1 ( \R^d))$ in this way by smooth Green functions that pointwise converge to it \cite[Lemma 7]{aronson1968non}.
An alternative regularisation argument is given by \cite[Proof of Lemma 2.4]{bogachev2016distances}. 
\end{proof}

Let us first estimate the error near the terminal time. 

\begin{lemma}[Terminal time comparison] \label{lem:TemrCont} Assuming \ref{ass:H} without requiring the semiconcavity for $g$ in \ref{ass:H}-\ref{item:ATerm},
there exists $C>0$ such that for every $(\varepsilon,\tau) \in (0,1/2] \times [0,T]$,
\[ \sup_{(t,x) \in [T-\tau,T] \times \R^d} | \varphi_t^{\varepsilon}(x) - \varphi_t^0(x)| \leq C \tau +  2 \lVert \nabla g \rVert_{L^\infty} \sqrt{ \varepsilon \tau}. \]
\end{lemma}

\begin{proof}
We use a standard comparison argument.
Let $(S^\varepsilon_t)_{t \geq 0}$ denote the heat semigroup satisfying $\partial_t S^\varepsilon_t = \tfrac{\varepsilon}{2} \Delta S^\varepsilon_t$.
Using \eqref{eq:GradBound} and \ref{ass:H}-\ref{item:ALip}, $\H(x,\nabla \varphi^\varepsilon_t(x
))$ is bounded uniformly in $\varepsilon$. 
Since $g$ is Lipschitz, $\nabla S^\varepsilon_t [g]$ is bounded uniformly in $\varepsilon$.
We can thus find $C > 0$ independent of $\varepsilon$ such that $\psi^{\varepsilon,\pm} : t \mapsto S^\varepsilon_{T-t} [g] \pm C (t-T) $ is a super/sub-solution of \eqref{eq:HJBeps}, so that $\psi^{\varepsilon,-}_t  \leq \varphi^\varepsilon_t \leq \psi^{\varepsilon,+}_t$ by comparison.
Using the explicit kernel for $S^\varepsilon$, we get that $\lVert S^\varepsilon_{T-t} [g] - g \rVert_{L^\infty} \leq \sqrt{\varepsilon(T-t)} ) \lVert \nabla g \rVert_{L^\infty}$. 

Since \eqref{eq:HJB0} is satisfied almost everywhere and $\varphi^0$ is Lipschitz-continuous in $x$ by \eqref{eq:Grad+SC}, we deduce that   $\varphi^0$ is also Lipschitz-continuous in $t$ with constant depending only on the Lipschitz constant of $g$ and the constants appearing in Assumption \ref{ass:H}-\ref{item:ALip}. As a consequence, for some $C>0$ and all $t \in [0,T]$, $ \lVert \varphi^0_t - g \rVert_{L^{\infty}} \leq C(T-t)$. 
The conclusion follows by triangular inequality.
\end{proof}

We are now ready to compute the optimal rate.

\begin{theorem}[Optimal rate] \label{thm:optSC}
Under \ref{ass:H}, there exists $C_{\log} >0$ such that for every $\varepsilon \in (0,1/2]$,
\[ \forall (t,x) \in [0,T] \times \R^d, \quad \varphi^\varepsilon_t (x) - \varphi^0_t (x) \geq d \varepsilon \log \varepsilon -C_{\log} \varepsilon. \]
\end{theorem}

\begin{proof}
Subtracting \eqref{eq:HJBeps} to \eqref{eq:convolSub}, a direct computation yields
\begin{equation} \label{eq:KeyLower}
\big[ \partial_s - b^{\varepsilon,\delta}_s \cdot \nabla + \tfrac{\varepsilon}{2} \Delta \big] ( \varphi^{\varepsilon}_s - \varphi^{0,\delta}_s ) \leq - \frac{\varepsilon}{2} \Delta \varphi^{0,\delta}_s + 2 \delta C_{\mathrm{sub}}. 
\end{equation} 
The regularity given by Lemma \ref{lem:supconv}-\ref{item:supReg} allows us to integrate \eqref{eq:KeyLower} along the flow $(\mu^{\varepsilon,\delta,x}_s)_{t \leq s \leq T}$. 
Since $\varphi^\varepsilon_T = \varphi^{0}_T = g$, this yields
\begin{equation} \label{eq:StepBefore}
\forall (t,x) \in [0,T] \times \R^d, \quad \varphi^\varepsilon_t (x) - \varphi^{0,\delta}_t (x) \geq - \lVert g - g^\delta \rVert_\infty - 2 \delta C_{\mathrm{sub}} + \frac{\varepsilon}{2} \int_t^T \Delta \varphi^{0,\delta}_s \d \mu^{\varepsilon,\delta,x}_s \d s. 
\end{equation} 
From Lemma \ref{lem:supconv}-\ref{item:supReg}, $\varphi^{0,\delta}$ is $-1/\delta$-semiconvex, so that for $\tau \in (0,T-t]$,
\begin{equation} \label{eq:LastConvS}
\frac{\varepsilon}{2} \int_t^{t+\tau} \int_{\R^d} \Delta \varphi^{0,\delta}_s \d \mu^{\varepsilon,\delta,x}_s  \d s \geq - \frac{\varepsilon \tau d}{2\delta}. 
\end{equation} 
Proposition \ref{pro:Regularising}-\ref{item:LapLow} and Lemma \ref{lem:supconv}-\ref{item:supAppr} then yield
\[ \varphi^\varepsilon_t (x) - \varphi^{0}_t (x) \geq - 2 \delta ( 2L + C_{\mathrm{sub}}) - \frac{\varepsilon \tau d}{2\delta} + \frac{\varepsilon d}{2} \log ( 2 \pi \varepsilon \tau ) - \frac{\tau L^2}{2} - C \varepsilon, \]
for every $(t,x) \in [0,T] \times \R^d$. 
If $T - t \geq \varepsilon$, we can choose $\delta = \tau = \varepsilon$ and conclude.
Otherwise, Lemma \ref{lem:TemrCont} applied to $\eta = \varepsilon$ completes the proof.
\end{proof}

\begin{rem}[Recovering Theorem \ref{thm:subopt}] \label{rem:directBoundsub}
Since $\varphi^{0,\delta}_s$ is $-1/\delta$-semiconvex, we could directly lower bound $\Delta \varphi^{0,\delta}_s$ by $-d / \delta$ in \eqref{eq:StepBefore} and set $\delta = \sqrt{\varepsilon}$. 
This would recover the sub-optimal rate from Theorem \ref{thm:subopt}.
\end{rem}

\begin{rem}[Integrated Fisher information]
It is tempting to exploit the sign of the integrated Fisher information that appears in \eqref{ItoLog} after taking expectations.
For instance, we could integrate the Laplacian by parts in the last term of \eqref{eq:StepBefore}.
However, such an approach would yield the sub-optimal rate $\sqrt{\varepsilon}$, because of the $\varepsilon$ factor in front of the $\vert \nabla \log \mu^{\varepsilon,\delta}_s \vert^2$ term in \eqref{ItoLog}.   
\end{rem}

\section{Quadratic Hamiltonian}

This section is concerned with the quadratic case $\H(x,p) = \tfrac{1}{2} \vert p \vert^2$ and $g$ Lipschitz-continuous.
In this setting, a direct computation
(Cole-Hopf transform) shows that $-\varepsilon \log \varphi^\varepsilon_t$ solves a time-reversed heat equation. 
As a consequence, we get the the explicit formula
\begin{equation} \label{eq:HopfC}
\forall (t,x) \in \R^d, \qquad \varphi^\varepsilon_t (x) = \frac{\varepsilon d}{2} \log ( 2 \pi \varepsilon(T-t) ) - \varepsilon \log \int_{\R^d} \exp \bigg[ - \frac{g(y)}{\varepsilon} - \frac{|y-x|^2}{2 \varepsilon (T-t)} \bigg] \d y. 
\end{equation} 
As $\varepsilon \rightarrow 0$, the Laplace principle yields the Hopf-Lax formula
\begin{equation} \label{eq:HopfLax}
\varphi^0_t (x) = \inf_{y \in \R^d} g(y) + \frac{1}{2(T-t)^2} \vert y - x \vert^2. 
\end{equation} 
These formulae have been used in \cite{qian2024optimal} to derive optimal rates in a one-dimensional setting. In Section \ref{subsec:NSC} we extend the result of Theorem \ref{thm:optSC} to merely Lipschitz-continuous terminal data. 
Section \ref{sec:Example} eventually shows the optimality of the $O(\varepsilon \log \varepsilon)$ rate on an explicit example.

\subsection{Lipschitz terminal data} \label{subsec:NSC}

In this section, we assume that $g$ is Lipschitz-continuous  with Lipschitz constant $L$ but not necessarily semi-concave. 
We rely on the ``concavifying'' effect of Hamilton-Jacobi equations \cite{cannarsa2004semiconcave}, which is somehow inherited from the parabolic regularity through \eqref{eq:HopfC}. 

\begin{lemma}[Semiconcavity generation] \label{lem:SC+}
For $(\varepsilon,t) \in (0,1/2] \times [0,T)$, $\varphi^\varepsilon_t$ is $\tfrac{1}{T-t}$-semiconcave.
\end{lemma}

\begin{proof}
From \eqref{eq:HopfC}, for $(t,x) \in [0,T] \times \R^d$,
\begin{multline*}
\nabla^2 \varphi^\varepsilon_t (x) = \frac{1}{T-t} \mathrm{Id} - \frac{\int_{\R^d} (y-x) \otimes (y-x) \exp \big[ - \frac{g(y)}{\varepsilon} - \frac{|y-x|^2}{2 \varepsilon (T-t)} \big] \d y}{\varepsilon (T-t)^2 \int_{\R^d} \exp \big[ - \frac{g(y)}{\varepsilon} - \frac{|y-x|^2}{2 \varepsilon (T-t)} \big] \d y} \\
+ \frac{1}{\varepsilon(T-t)^2} \bigg( \frac{\int_{\R^d} (y-x) \exp \big[ - \frac{g(y)}{\varepsilon} - \frac{|y-x|^2}{2 \varepsilon (T-t)} \big] \d y}{ \int_{\R^d} \exp \big[ - \frac{g(y)}{\varepsilon} - \frac{|y-x|^2}{2 \varepsilon (T-t)} \big] \d y} \bigg) \otimes \bigg( \frac{\int_{\R^d} (y-x) \exp \big[ - \frac{g(y)}{\varepsilon} - \frac{|y-x|^2}{2 \varepsilon (T-t)} \big] \d y}{ \int_{\R^d} \exp \big[ - \frac{g(y)}{\varepsilon} - \frac{|y-x|^2}{2 \varepsilon (T-t)} \big] \d y} \bigg). 
\end{multline*} 
The last two terms gather to give the opposite of a covariance matrix, thus a negative semidefinite matrix, proving the result. 
\end{proof}

This property extends to $\varepsilon = 0$, which is the analogous of Lemma \ref{lem:supconv}-\ref{item:supReg} for inf-convolution. 
We now establish the analogous of Lemma \ref{lem:UppOpt}.

\begin{proposition}[Upper bound] \label{pro:UppNSC}
For every $\varepsilon \in (0,1/2]$,
\begin{equation*} \label{eq:Upper}
\forall (t,x) \in [0,T] \times \R^d, \quad \varphi^{\varepsilon}_t ( x) - \varphi^{0}_t ( x) \leq - \frac{d }{2} \varepsilon \log \varepsilon + C \varepsilon,
\end{equation*}
for some $C>0$ depending only on $(\lVert \nabla g \rVert_{L^{\infty}},T,d)$.
\end{proposition}

\begin{proof}
We first fix $\eta \in (0,T)$ and we verify that, because of the $\frac{1}{T-t}$-semi-concavity of $\varphi^0_t$, $(t,x) \mapsto \varphi^0_t (x) + \tfrac{\varepsilon}{2} \int_t^{T-
\eta} \frac{d}{T-s} \d s$ is a viscosity super-solution of 
    \[ -\partial_t \psi_t + \frac{1}{2} \vert \nabla \psi_t \vert^2 - \frac{\varepsilon}{2} \Delta \psi_t = 0, \quad \mbox{ in  } [0,T-\eta] \times \R^d, \quad \psi_{T-\eta} = \varphi^0_{T-\eta}, \]
and we deduce that
\begin{align*} 
\sup_{t \in [0,T-\eta] \times \R^d} \varphi^{\varepsilon}_t(x) - \varphi^0_t(x) &\leq \frac{d\varepsilon}{2} \int_0^{T-\eta} \frac{1}{T-t} \d t + \sup_{x \in \R^d} \vert \varphi^{\varepsilon}_{T-\eta}(x) - \varphi^0_{T-\eta} \vert  \\
& \leq \frac{d\varepsilon}{2} \log \frac{T}{\eta} + \sup_{(t,x) \in [T-\eta,T] \times \R^d} \vert \varphi_t^{\varepsilon}(x) - \varphi^0_t(x) \vert.  
\end{align*}
We use Lemma \ref{lem:TemrCont}, which does not rely on the semi-concavity of the terminal data, to estimate the difference near $T$. We pick $\eta = \varepsilon$ to conclude.
\end{proof}

We are now ready to prove the analogous of Theorem \ref{thm:optSC}.

\begin{proposition}[Lower bound]\label{pro:LowerNSC}
For every $(t,x) \in (0,T] \times \R^d$, for every $\varepsilon \in (0,1/2]$,
\[ \varphi^\varepsilon_t (x) - \varphi^{0}_t (x) \geq  -d \varepsilon \log \varepsilon - C\varepsilon, \]
for a constant $C > 0$ independent of $\varepsilon >0$.
\end{proposition}

\begin{proof}[Proof of Proposition \ref{pro:LowerNSC}]
Let $(t,x) \in [0,T] \times \R^d$. If $T-t \leq 2 \varepsilon$ we simply use Lemma \ref{lem:TemrCont}. So we assume that $2\varepsilon < (T-t)$.

We only perform a small change in the proof of  Theorem \ref{thm:optSC}: we integrate \eqref{eq:KeyLower} along the flow $(\mu_s^{\varepsilon,\delta})_{t \leq s \leq T}$ between time $t$ and $T - \eta$ for $\delta < \eta \leq T-t$. 
We have $C_{\mathrm{sub}} =0$ in Lemma \ref{lem:supconv}-\ref{item:subsol} because $\H$ no more depends on $x$.
This replaces \eqref{eq:StepBefore} by
\begin{equation} \label{eq:StepNSC}
\varphi^\varepsilon_t (x) - \varphi^{0,\delta}_t (x) \geq \int_{\R^d} [ \varphi^\varepsilon_{T-\eta} - \varphi^{0,\delta}_{T-\eta} ] \d \mu^{\varepsilon,\delta,x}_{T-\eta} + \frac{\varepsilon}{2} \int_t^{T-\eta} \int_{\R^d} \Delta \varphi^{0,\delta}_s \d \mu^{\varepsilon,\delta,x}_s \d s. 
\end{equation}
Using the Hop-Lax formula \eqref{eq:HopfLax}, since $\sup_z \inf_y \leq \inf_y \sup_z$, we notice that
\[ \varphi^{0,\delta}_{T-\eta} (x) = \sup_{z \in \R^d} \inf_{y \in \R^d} g(y) + \frac{1}{2 \eta} \vert z - y \vert^2 - \frac{1}{2 \delta} \vert x - z \vert^2 \leq \varphi^{0}_{T-\eta + \delta} (x). \]
Lemma \ref{lem:TemrCont} and the time Lipschitz regularity of $\varphi^0$ then implies
\begin{align*}
\int_{\R^d} [ \varphi^\varepsilon_{T-\eta} - \varphi^{0,\delta}_{T-\eta} ] \d \mu^{\varepsilon,\delta,x}_{T-\eta} &\geq - \lVert \varphi^\varepsilon_{T-\eta} - \varphi^0_{T-\eta} \rVert_{L^\infty} - \lVert \varphi^0_{T-\eta} - \varphi^{0}_{T-\eta+\delta} \rVert_{L^\infty} \geq - C \eta - C \sqrt{\varepsilon \eta} - C \delta .  
\end{align*} 
For $\tau \in (0,T-t-\eta]$, Lemma \ref{lem:SC+} gives that 
\[ -\frac{1}{2} \int_{t+\tau}^{T-\eta} \int_{\R^d} \Delta \varphi^{\varepsilon}_s \d \mu^{\varepsilon,\delta}_s \d s \geq - \frac{1}{2} \int_{t+\tau}^{T-\eta}  \frac{d}{2 (T-s)} \d s \geq \frac{d}{2} \log \frac{\eta}{T}. \]
After simplifications, this replaces the result of Proposition \ref{pro:Regularising}-\ref{item:LapLow} by
\[ \frac{1}{2} \int_{t+\tau}^{T-\eta} \int_{\R^d} \Delta \varphi^{0,\delta}_s \d \mu^{\varepsilon,\delta, x}_s \d s \geq \frac{d}{2} \log ( 2 \pi \varepsilon \tau ) - \frac{\tau}{2 \varepsilon} L^2 + \frac{d}{2} \log \frac{\eta}{T}  - C, \]
which we can plug in \eqref{eq:StepNSC} as previously, together with \eqref{eq:LastConvS}.
Gathering pieces now yields
\[ \varphi^\varepsilon_t (x) - \varphi^{0}_t (x) \geq  - C( \varepsilon + \delta) - C \sqrt{ \varepsilon \delta}  - \tfrac{\varepsilon \tau d}{2\delta} + \tfrac{\varepsilon d}{2} \log ( 2 \pi \varepsilon \tau ) - \tfrac{\tau L^2}{2} + \tfrac{\varepsilon d}{2} \log \tfrac{\eta}{T} , \]
for every $(t,x) \in [0,T-\eta] \times \R^d$. We conclude by choosing $\tau = \eta = \varepsilon$ and $\delta \sim \varepsilon$. 
\end{proof}

\subsection{Explicit solutions and optimality of the rate} \label{sec:Example}

In this section, we provide an example to justify that the rate $O(\varepsilon \log \varepsilon)$ cannot be improved, at least in dimension $d \geq 2$. 
In dimension $d=1$ a sharp example is provided in \cite[Proposition 4.4]{qian2024optimal}.
We rely on the explicit formula \eqref{eq:HopfC} and the Hopf-Lax formula \eqref{eq:HopfLax}, which corresponds to the $\varepsilon \rightarrow 0$ limit of \eqref{eq:HopfC}. 

\begin{proposition} \label{pro:example}
    For $1 \leq k  \leq d$ and $x= (x_1, \cdots,x_d) \in \R^d$, define the orthogonal  projection $P_k(x) := (x_1,\dots,x_k,0, \dots,0)$ into the first $k$ coordinates. Let $\varphi^{k,\varepsilon}$ and $\varphi^{k,0}$ be respectively the solutions to \eqref{eq:HJBeps} and \eqref{eq:HJB0} with terminal data $g_k = - |P_k(x)|$ and source term $f \equiv 0$. Then, for all $T >0$ and all $t \in [0,T)$ we have the expansion
     $$ \varphi^{k,\varepsilon}_t(0) = \varphi^{k,0}_t(0) + \frac{k-1}{2} \varepsilon \log \varepsilon  - \frac{k-1}{2} \varepsilon \log (T-t) - \varepsilon \log \frac{k \sqrt{\pi}}{ 2^{\frac{k-1}{2}} \Gamma(\frac{k}{2}+1)  } + o(\varepsilon).$$
\end{proposition}

\begin{proof}
Using the Hopf-Lax formula, for all $(\tau,x) \in (0,T] \times \R^d$ we have
\[ \varphi_{T-\tau}^{k,0}(x) = \inf_{y \in \R^d} \bigl \{ -|P_k(y)| + \frac{1}{2\tau} |x-y|^2 \bigr \} = -|P_k(x)| - \frac{\tau}{2}.
\]
Moreover, the infimum is obtained for $y =  x + \tau \frac{P_k(x)}{\vert P_k (x) \vert}$ if $ P_k(x) \neq 0$ and for any $y = x + z$ with $z \in \mathrm{Span}(e_1,\dots,e_k)$ and $\vert z \vert = \tau$
such that $|P_k(y)| = \tau$ when $ P_k (x) = 0$. On the other hand, using \eqref{eq:HopfC} for all $(\tau,x) \in (0,T] \times \R^d$, 
\[ \varphi_{T-\tau}^{k, \varepsilon}(x) =  - \varepsilon \log \Bigl[ \int_{\R^d} e^{\frac{-1}{\varepsilon} \bigl( -|P_k(z+x)|+ \frac{1}{2 \tau} | z|^2 \bigr)} \frac{\d z}{(2 \pi \varepsilon \tau)^{d/2}} \Bigr] \]
Now we decompose $|z|^2 = |P_k(z)|^2 + |z-P_k(z)|^2$ and introduce the two variables $z' = (z_1,\dots,z_k)$ and $z''=(z_{k+1},\dots,z_d)$ to compute, with the slight abuse of notation $P_k(x) = (x_1, \dots,x_k) \in \R^k$,
\begin{align*}
 \varphi_{T-\tau}^{k, \varepsilon}(x) &= - \varepsilon \log \Bigl[ \int_{\R^d} e^{\frac{-1}{\varepsilon} \bigl( -|P_k(z+x)|+ \frac{1}{2 \tau} | P_k(z)|^2 \bigr)} e^{-\frac{1}{2\varepsilon\tau} |z-P_k(z)|^2} \frac{\d z}{(2 \pi \varepsilon \tau)^{d/2}} \Bigr] \\
    &= - \varepsilon \log \Bigl[ \int_{\R^k}  e^{\frac{-1}{\varepsilon} \bigl( -|P_k(x) +z'|+ \frac{1}{2 \tau} | z'|^2 \bigr)} \frac{\d z'}{(2 \pi \varepsilon \tau)^{k/2}} \int_{\R^{d-k}} e^{-\frac{1}{2\varepsilon\tau} |z''|^2} \frac{\d z''}{(2 \pi \varepsilon \tau)^{(d-k)/2}}\Bigr] \\
    &=- \varepsilon \log \Bigl[ \int_{\R^k}  e^{\frac{-1}{\varepsilon} \bigl( -|z'|+ \frac{1}{2 \tau} | z' - P_k(x)|^2 \bigr)} \frac{\d z'}{(2 \pi \varepsilon \tau)^{k/2}} \Bigr].
\end{align*}

For any $\tau \in (0,T]$, we expect that the rate deteriorates when $P_k (x) = 0$, since this is precisely where the map $ y \mapsto -|P_k(y)| + \frac{1}{2\tau}|x-y|^2$ has several minima. From now on we take $x=0$. We have $ \varphi^0_{T-\tau}(0) = -\tau/2$, while rewriting $-|y| + \frac{1}{2\tau}|y|^2 = \frac{1}{2\tau}(|y|-\tau)^2 - \frac{\tau}{2}$,
we get
$$ \varphi_{T-\tau}^{k,\varepsilon} (0) = - \frac{\tau}{2} + \frac{k}{2} \varepsilon \log ( 2 \pi \varepsilon\tau) -    \varepsilon \log \int_{\R^k} e^{-\frac{1}{\varepsilon} \frac{1}{2 \tau} (|y|- \tau)^2 } \d y . $$
By a change of variables for radial functions and the change of variable $s = \tfrac{r-\tau}{\sqrt{\varepsilon\tau}}$, we get
$$  \int_{\R^{k}}  e^{-\frac{1}{\varepsilon} \frac{1}{2 \tau} (|y|- \tau)^2 } \d y = C_k \int_0^{\infty} e^{ - \frac{1}{\varepsilon} \frac{1}{2 \tau} (r-\tau)^2 } r^{k-1} \d r = C_k \sqrt{\varepsilon \tau} \int_{ - \frac{\sqrt{\tau}}{\sqrt{\varepsilon}}}^{+\infty} e^{-s^2/2}( \sqrt{\varepsilon \tau} s + \tau)^{k-1} \d s, $$  
where $C_k = k \pi^{k/2} / \Gamma(k/2+1)$ is $k$ times the Lebesgue measure of the unit ball of $\R^k$. We then easily see that 
$$ \lim_{\varepsilon \rightarrow 0^+} \int_{ - \frac{\sqrt{\tau}}{\sqrt{\varepsilon}}}^{+\infty} e^{-s^2/2}( \sqrt{\varepsilon \tau} s + \tau)^{k-1} \d s = \sqrt{2\pi} \tau^{k-1}, $$
from which the result follows.
\end{proof}

\section*{Acknowledgements}
\addcontentsline{toc}{section}{Acknowledgement}

The first author is grateful to Daniel Lacker for suggesting this short probabilistic proof of Proposition \ref{pro:Regularising}-\ref{item:RegEnt}.
The authors also thank Marco Cirant and Alessandro Goffi for suggesting some generalisations of the results in the first version of this work.

\printbibliography
\addcontentsline{toc}{section}{References}

\end{document}